\documentclass[12pt,a4paper,twoside,english,american]{scrartcl}
\usepackage{helvet}
\usepackage[latin9]{inputenc}
\usepackage{amsmath}
\usepackage{amssymb}
\usepackage{esint}

\makeatletter

\usepackage{a4wide}
\usepackage{tu-preprint}
\usepackage{mathtools}
\usepackage{url}

\usepackage{euscript}\usepackage{mathabx}
\allowdisplaybreaks[1] 
\usepackage{amsfonts}
\newenvironment{keywords}{ \noindent\footnotesize\textbf{Keywords and phrases:}}{}

\newenvironment{class}{\noindent\footnotesize\textbf{Mathematics subject classification 2010:}}{}
\usepackage{amsthm}

 \theoremstyle{plain}
\newtheorem{thm}{Theorem}[section]
  \theoremstyle{definition}
  \newtheorem{defn}[thm]{Definition}
  \theoremstyle{remark}
  \newtheorem{rem}[thm]{Remark}
 \theoremstyle{definition}
  \newtheorem{example}[thm]{Example}
  \theoremstyle{plain}
  \newtheorem{prop}[thm]{Proposition}
  \theoremstyle{plain}
  
  \theoremstyle{plain}
  \newtheorem*{lem*}{Lemma}
  \theoremstyle{plain}
  \newtheorem*{thm*}{Theorem}
  \theoremstyle{plain}
  \newtheorem{cor}[thm]{Corollary}
  \theoremstyle{remark}

\@ifundefined{definecolor}
 {\usepackage{color}}{}
\usepackage{tu-preprint}

\newcommand{\Abs}[1]{\left\lVert#1\right\rVert}

\newcommand{\R}{\mathbb{R}}
\newcommand{\K}{\mathbb{K}}
\newcommand{\N}{\mathbb{N}}

\newcommand*{\abs}[1]{\left\lvert#1\right\rvert}
\newcommand*{\lip}[1]{\lvert#1\rvert_{\textnormal{Lip}}}

\newcommand*{\supp}{\operatorname{supp}}

\newcommand*{\dd}{\mathrm{d}}
\newcommand*{\1}{\chi}

\newcommand{\s}[1]{\mathcal{#1}}

\DeclareMathAccent{\Circ}{\mathalpha}{operators}{"17}

\newcommand{\eps}{\varepsilon}

\renewcommand{\tilde}{\widetilde}
\renewcommand*{\epsilon}{\varepsilon}
\renewcommand*{\rho}{\varrho}

\usepackage{babel}

\makeatother

\usepackage{babel}
\begin{document}
\selectlanguage{english}%
\institut{Institut f\"ur Analysis}

\preprintnumber{MATH-AN-11-2012}

\preprinttitle{A Functional Analytic Perspective to Delay Differential Equations}

\author{Rainer Picard, Sascha Trostorff, Marcus Waurick}

\makepreprinttitlepage\setcounter{section}{-1}

\date{}

\selectlanguage{american}%

\title{A Functional Analytic Perspective to Delay Differential Equations}

\author{Rainer Picard,\\
 Sascha Trostorff \& Marcus Waurick\foreignlanguage{english}{\thanks{Institut für Analysis,Fachrichtung Mathematik, Technische Universität Dresden, Germany, rainer.picard@tu-dresden.de, sascha.trostorff@tu-dresden.de,  marcus.waurick@tu-dresden.de }}}
\maketitle
\begin{abstract}
\textbf{Abstract.} We generalize the solution theory for a class of
delay type differential equations developed in a previous paper, dealing
with the Hilbert space case, to a Banach space setting. The key idea
is to consider differentiation as an operator with the whole real
line as the underlying domain as a means to incorporate pre-history
data. We focus our attention on the issue of causality of the differential
equations as a characterizing feature of evolutionary problems and
discuss various examples. The arguments mainly rely on a variant of
the contraction mapping theorem and a few well-known facts from functional
analysis.
\end{abstract}

\begin{class} 34K05 (Functional differential equations, general theory);
34K30 (Equations in abstract spaces); 34K40 (Neutral equations); 34A12
(Initial value problems, existence, uniqueness, continuous dependence
and continuation of solutions); 45G15 (Systems of nonlinear integral
equations) \end{class}

\newpage

\tableofcontents{}

\newpage

\section{Introduction}

In this note, we present a unified Banach space perspective to a class
of ordinary differential equations with memory and delay effects.
This class is often summarized under the umbrella term delay differential
equations: 
\[
\dot{x}(t)=f(t,x_{t},\dot{x}_{t})\ t\in I,\ I\subseteqq\R\ \text{interval},
\]
where $\dot{x}$ denotes the (time-)derivative of $x$ and $x_{t}$
denotes the \emph{history segment of $x$}, cf. e.g. \cite{DieGilVerWal,HalLun1993}
see also Example \ref{Continuous-Delay} below. These equations have
many applications in engineering or sciences. We refer to \cite{Bal2009,DiekGyl2008}
and the references therein for an account of various applications. 

The class considered here covers ordinary differential equations,
differential difference equations, integro-differential equations
or even neutral differential equations. Using some basic functional
analysis, the main contribution of this note is that the aforementioned
equations can be treated in a unified manner. The core idea consists
in the treatment of the problem on the whole real axis as time-line.
This enables us to conveniently detour the introduction of certain
(pre-)history spaces, cf. e.g. \cite{DieGilVerWal,HinoMurakamiNaito}.
Moreover, we do not need to introduce an extended state space as it
can be done for linear theory using semi-group methods, see e.g. \cite{BatPia}.

Our perspective also shows that finite and infinite delay do not need
different treatment. However, our focus is on existence of solutions
and continuous dependence on the data rather than the continuity or
differentiability or other qualitative properties of the solution
itself. This yields an \emph{elementary} solution theory at least
for $L_{p}$-spaces. Though parts of our results are covered by well-known
theory, we have more freedom in choosing right-hand sides. In particular,
we can solve delay differential equations with certain measures or
functions with unbounded variation as right-hand sides, cf. Subsection
\ref{sub:Examples}. 

Our development of a solution theory for delay differential equations
starts in Section \ref{sec:gst}, where we state a variant of the
contraction mapping theorem tailored for operator equations in the
abstract form 
\[
Cx=F(x),
\]
for $x$ residing in some Banach space $X$ and $F$ being Lipschitz-continuous
in suitable sense and $C:D(C)\subseteqq X\to X$ being a closed, densely
defined, continuously invertible linear operator. The remaining parts
of this paper are essentially applications of the results from Section
\ref{sec:gst}. To this end, we have to establish the (time-)derivative
as a \emph{continuously invertible }operator\emph{. }Thus, both the
Sections\emph{ }\ref{sec:dde-lp} and \ref{sec:The-non-reflexive-case}
start with the definition of the time-derivative as a continuously
invertible operator in a $L_{p}$-space and in a space of continuous
functions. Having our main applications of delay differential equations
in mind, we introduce the (time-)derivative on functions defined on
the whole real line as time axis. If one wants to recover classical
theory, e.g. initial value problems for ordinary differential equations,
one has to know that the solution of a differential equation depends
only on the past of the right-hand side. The latter is summarized
by the notion of causality, see \cite{Laksh10} or  Definition \ref{As-causality-is}
below. Thus, in both the Section \ref{sec:dde-lp} and \ref{sec:The-non-reflexive-case},
we also show causality of the respective solution operators. These
sections are complemented by the discussion of several examples. 

It should be noted that many ideas rely on results and strategies
used in reference \cite{Kalauch}, where a Hilbert space perspective
is preferred. The idea of introducing the time-derivative as a continuously
invertible operator stems from references \cite{PicTrans89,Pi2009-1,PicMcG2011},
which in view of the $L_{\infty}$-considerations in Section\foreignlanguage{english}{
\ref{sec:The-non-reflexive-case} }shows its kinship to ideas developed
back in 1952 by Morgenstern, \cite{Ref195229}.

\section{The general solution theory - a variant of the contraction mapping
theorem}

\label{sec:gst}

Let $X$ be a Banach space and let $C:D(C)\subseteqq X\to X$ be a
densely defined closed linear operator with $0\in\rho(C)$. Then $\s X_{1}(C):=(D(C),\abs{C\cdot}_{X})$
is a Banach space. Moreover, define $\s X_{-1}(C)$ to be the completion
$(X,\abs{C^{-1}\cdot}_{X})^{\sim}$ of $(X,\abs{C^{-1}\cdot}_{X})$
and let $\s X_{0}(C):=X$. If the operator $C$ is clear from the
context, we omit the reference to the operator $C$ in the notation
of the associated spaces. We have 
\[
\s X_{1}\hookrightarrow\s X_{0}\hookrightarrow\s X_{-1}
\]
in the sense of continuous and dense embedding. Furthermore, $C:D(C)\subseteqq\s X_{0}\to\s X_{-1}$
is isometric and densely defined with dense range. Hence, $C$ can
be extended to an isometric isomorphism. We identify $C$ with its
extension. The fundamental solution theory is based on the following
variant of the contraction mapping theorem.

\begin{thm}\label{thm:bst} Let $F:\s X_{0}\to\s X_{-1}$ be a strict
contraction and let $\lip{F}(<1)$ be the best Lipschitz constant
for $F$. Then the equation 
\[
Cx=F(x)
\]
has a unique fixed point $x\in\s X_{0}$. If $y\in\s X_{0}$ and $n\in\N$
then the following estimates hold
\begin{align*}
\abs{C^{-1}F(y)-x} & \leqq\frac{\lip{F}}{1-\lip{F}}\abs{y-x},\\
\abs{\left(C^{-1}F\right)^{n}(y)-x} & \leqq\frac{\lip{F}}{1-\lip{F}}\abs{\left(C^{-1}F\right)^{n}(y)-\left(C^{-1}F\right)^{n-1}(y)},\\
\abs{\left(C^{-1}F\right)^{n}(y)-x} & \leqq\frac{\lip{F}^{n}}{1-\lip{F}}\abs{y-x}.
\end{align*}
 \end{thm} \begin{proof} The operator $C^{-1}:\s X_{-1}\to\s X_{0}$
is an isometric isomorphism. Hence, $C^{-1}F(\cdot)$ is a strict
contraction in $\s X_{0}$. The contraction mapping theorem yields
the assertion. The estimates are well-known. \end{proof}

\begin{cor} Let $F:\s X_{1}\to\s X_{0}$ be a strict contraction.
Then 
\[
Cx=F(x)
\]
has a unique fixed point $x\in\s X_{1}$. \end{cor} \begin{proof}
The mapping $CF(C^{-1}\cdot)$ satisfies the assumptions from Theorem
\ref{thm:bst}. Hence, there exists a unique fixed point $\tilde{x}\in\s X_{0}$
such that 
\[
C\tilde{x}=CF(C^{-1}\tilde{x}).
\]
 Therefore $x:=C^{-1}\tilde{x}\in\s X_{1}$ satisfies 
\[
Cx=F(x).
\]
 Now, let $u\in\s X_{1}$ satisfy $Cu=F(u)$. Then $Cu$ satisfies
the equation $C(Cu)=CF(u)=CF(C^{-1}(Cu))$ and thus, $Cu=\tilde{x}$,
which gives $u=x$. \end{proof}

\begin{thm} Let $F,G:\s X_{0}\to\s X_{-1}$ be Lipschitz continuous
and assume that the respective Lipschitz semi-norms, i.e., the best
Lipschitz constants, $\lip{F}$ and $\lip{G}$ satisfy 
\[
\frac{\lip{F}+\lip{G}}{2}<1.
\]
 Let $x,y\in\s X_{0}$ satisfy 
\[
Cx=F(x)\text{ and }Cy=G(y).
\]
 Then 
\[
\abs{x-y}_{\s X_{0}}\leqq\frac{1}{1-\frac{\lip{F}+\lip{G}}{2}}\sup_{u\in\s X_{0}}\abs{F(u)-G(u)}_{\s X_{-1}}.
\]
 \end{thm} \begin{proof} By assumption, we have 
\begin{align*}
x-y & =C^{-1}F(x)-C^{-1}G(y)\\
 & =\frac{1}{2}C^{-1}(F(x)-F(y))\\
 & \quad+\frac{1}{2}C^{-1}(G(x)-G(y))-\frac{1}{2}C^{-1}(G(x)-F(x))+\frac{1}{2}C^{-1}(F(y)-G(y)).
\end{align*}
 This yields the assertion. \end{proof}

\section{The reflexive case -- delay differential equations in $L_{p}$-spaces}

\label{sec:dde-lp}

For the whole section, let $p\in(1,\infty)$ and denote by $q$ the
conjugate exponent such that $\frac{1}{p}+\frac{1}{q}=1$. Let $X$
be a Banach space.

\subsection{Definition of the time-derivative}

Denoting by $\mu_{\nu}$ the weighted Lebesgue measure on $\R$ with
Radon-Nikodym derivative $x\mapsto e^{-\nu px}$ for $\nu\in\R$,
we define 
\[
W_{p,\nu}^{0}(\R;X):=L_{p,\nu}(\R;X):=L_{p}(\mu_{\nu};X).
\]
 Note that the mapping $e^{-\nu m}:L_{p,\nu}(\R;X)\to L_{p}(\R;X),f\mapsto\left(x\mapsto e^{-\nu x}f(x)\right)$
is isometrically isomorphic%
\footnote{The $m$ in the expression $e^{-\nu m}$ serves as reminder of multiplication.
We will frequently use this notation. For instance, let $\phi\colon\R\to\R$
and $\psi\colon\R\to E$ for some vector space $E$. Then we define
$\phi(m)\psi$ to be the mapping 
\[
\phi(m)\psi\colon\R\to E,t\mapsto\phi(t)\psi(t).
\]
}.

\begin{defn} We define 
\[
\partial_{\nu}:C_{c}^{\infty}(\R;X)\subseteqq L_{p,\nu}(\R;X)\to L_{p,\nu}(\R;X),f\mapsto f',
\]
where 
\[
C_{c}^{\infty}(\R;X):=\{\phi;\phi\text{ indefinitely differentiable, }\supp\phi\text{ compact}\}.
\]
 \end{defn}

The operator $\partial_{\nu}$ is clearly closable and its closure
coincides with the distributional derivative. Henceforth, we will
not distinguish notationally between $\partial_{\nu}$ and its closure.
In order to apply the general solution theory to $\partial_{\nu}$
in place of $C$, we need the following theorem:

\begin{thm} Assume $\nu>0$. Then we have that the convolution operator
$\1_{[0,\infty)}\ast:L_{p,\nu}(\R;X)\to L_{p,\nu}(\R;X)$ is continuous
with operator norm equal to $\frac{1}{\nu}$. Moreover, it holds $\left(\1_{[0,\infty)}\ast\right)^{-1}=\partial_{\nu}$.
\end{thm} 

\begin{proof} Let $\phi\in C_{c}^{\infty}(\R;X)$. Then we have,
invoking Young's inequality, that 
\begin{align*}
\abs{\1_{[0,\infty)}\ast\phi}_{p,\nu} & =\abs{e^{-\nu m}\left(\1_{[0,\infty)}\ast\phi\right)}_{p,0}\\
 & =\abs{\left(e^{-\nu m}\1_{[0,\infty)}\right)\ast\left(e^{-\nu m}\phi\right)}_{p,0}\\
 & \leqq\int_{\R}\abs{e^{-\nu t}\1_{[0,\infty)}(t)}\dd t\abs{e^{-\nu m}\phi}_{p,0}\\
 & =\frac{1}{\nu}\abs{\phi}_{p,\nu}.
\end{align*}
Now, for $n\in\N$ and some $x\in X$ with $\abs{x}=1$ define $\phi_{n}(t)\coloneqq\frac{1}{n^{1/p}}e^{\nu m}\1_{[0,n]}(t)x$
for all $t\in\R$. For $n\in\N$ let $u_{n}\coloneqq\1_{[0,\infty)}\ast\phi_{n}$.
It is easy to that $u_{n}(t)=\frac{1}{\nu n^{1/p}}\left(e^{\nu\min\{t,n\}}-1\right)x$
for all $t\in\R$ and that $\abs{\phi_{n}}=1$ for all $n\in\N$.
Moreover, an easy computation shows that $\abs{u_{n}}\to\frac{1}{\nu}$
as $n\to\infty$, for the details we refer to \cite[Proposition 2.2]{Trostorff}.

The equality $\left(\1_{[0,\infty)}\ast\right)^{-1}=\partial_{\nu}$
follows by differentiation of the convolution integral. \end{proof}

\begin{rem} If $\nu<0$, then a similar result holds. The respective
inverse, however, is now given by $\left(-\1_{(-\infty,0]}\ast\right)^{-1}$.
\end{rem}

Now, we are in the situation of our general solution theory with $C=\partial_{\nu}$.

For convenience, we describe the space $\s X_{-1}(\partial_{\nu})$
in more detail. We let $W_{p,\nu}^{1}(\R;X):=\s X_{1}(\partial_{\nu})$.
\begin{thm}\label{thm:dualspace} Assume that $X$ is reflexive%
\footnote{Note that, as a consequence, we have for any $\sigma$-finite measure
space $(\Omega,\mu)$ the property $L_{p}(\mu;X)'=L_{q}(\mu;X')$
(cf. \cite[p. 82: Corollary 4 and p.98: Theorem 1]{DiUh}). With the
help of \cite[p. 98: Theorem 1]{DiUh}, it thus suffices to assume
that $X'$ has the Radon-Nikodym property. %
}. We have 
\[
\left(W_{q,-\nu}^{1}(\R;X)\right)'=\left(L_{p,\nu}(\R;X');\abs{\partial_{\nu}^{-1}\cdot}_{p,\nu}\right)^{\tilde{}},
\]
 in the sense of the dual pairing 
\[
L_{p,\nu}(\R;X')\times L_{q,-\nu}(\R;X)\ni(\phi,\psi)\mapsto\intop_{\mathbb{R}}\langle\phi(t),\psi(t)\rangle_{X',X}\dd t=:\langle\phi,\psi\rangle_{0,0}.
\]
\par \end{thm} \begin{proof} Let $\phi\in L_{p,\nu}(\R;X')$ be
such that $\abs{\partial_{\nu}^{-1}\phi}=1$. Then, for $\psi\in W_{q,-\nu}^{1}(\R;X)$
with $\abs{\partial_{-\nu}\psi}_{q,-\nu}=1$ we have 
\begin{align*}
|\langle\phi,\psi\rangle_{0,0}| & =|\langle\partial_{\nu}\partial_{\nu}^{-1}\phi,\psi\rangle_{0,0}|\\
 & =|-\langle\partial_{\nu}^{-1}\phi,\partial_{-\nu}\psi\rangle_{0,0}|\\
 & =|\langle e^{-\nu m}\partial_{\nu}^{-1}\phi,e^{\nu m}\partial_{-\nu}\psi\rangle_{0,0}|\\
 & \leqq\abs{e^{-\nu m}\partial_{\nu}^{-1}\phi}_{p,0}\abs{e^{\nu m}\partial_{-\nu}\psi}_{q,0}=1.
\end{align*}
This establishes the continuity of $\iota\colon L_{p,\nu}(\R;X')\to W_{q,-\nu}^{1}(\R;X)',f\mapsto f$
with norm less than or equal to $1$. Now, since $L_{q,0}(\R;X)'=L_{p,0}(\R;X')$
by the reflexivity of $X$, we find $\tilde{\psi}$ in the unit ball
of $L_{q,0}(\R;X')$ such that $\langle e^{-\nu m}\partial_{\nu}^{-1}\phi,\tilde{\psi}\rangle_{0,0}=1$.
Defining $\psi:=-\partial_{-\nu}^{-1}(e^{\nu m})^{-1}\tilde{\psi}$,
we get that $\abs{\psi}_{q,-\nu,1}=1$. Thus, $\iota$ is isometric.
\par It remains to prove that $L_{p,\nu}(\R;X')$ is dense in $\left(W_{q,-\nu}^{1}(\R;X)\right)'$.
Let $\phi$ be a continuous linear functional on $\left(W_{q,-\nu}^{1}(\R;X)\right)'$
vanishing on $L_{p,\nu}(\R;X')$. By the reflexivity of $X$, we deduce
that $\phi\in W_{q,-\nu}^{1}(\R;X)$. Hence, $\phi=0$. \end{proof}

For our general solution theory the following corollary will be useful.
\begin{cor} Let $\ell\in\{1,0,-1\}$ and $\nu\in\R\setminus\{0\}$
and assume $X$ to be reflexive. Then, we have 
\[
W_{q,-\nu}^{-\ell}(\R;X')'=W_{p,\nu}^{\ell}(\R;X).
\]
 \end{cor} \begin{proof} The result is clear as a consequence of
Theorem \ref{thm:dualspace}. \end{proof}

\subsection{Solution theory}

We restate the basic solution theory in our particular situation.
However, we shall restrict ourselves to a particular form of right-hand
sides. We will need the following types of additional test function
spaces: 
\[
C_{c}^{\infty,+}(\R;X):=\{\phi;\phi\text{ indefinitely differentiable, }\sup\supp\phi<\infty,\exists n\in\N:\supp\phi^{(n)}\text{ compact}\}
\]
 and 
\[
C_{c}^{\infty,+}(\R;X)':=\{u:C_{c}^{\infty,+}(\R;X)\to\K;u\text{ linear}\}.
\]
We note here that we do not assume any specific continuity property
of the functionals in $C_{c}^{\infty,+}(\R;X)'$. The particular continuity
property will be assumed in the following definition.

\begin{defn}[eventually $(k,\ell)$-contracting] Let $k,\ell\in\{1,0,-1\}$
and let $Y$ be a reflexive Banach space. A mapping $F:C_{c}^{\infty}(\R;X)\to C_{c}^{\infty,+}(\R;Y')'$
is called \emph{eventually $(k,\ell)$-Lipschitz continuous} if the
following assumptions are satisfied: \begin{itemize} 

\item there exists $\nu_{0}>0$ such that for all $\nu\geqq\nu_{0}$
we have $F(0)\in W_{q,-\nu}^{-\ell}(\R;Y')'$, 

\item there exists $\nu_{1}>0$ and $C>0$ such that for all $\nu\geqq\nu_{1}$,
$u,v\in C_{c}^{\infty}(\R;X),\phi\in C_{c}^{\infty,+}(\R;Y')$ we
have 
\[
\abs{F(u)(\phi)-F(v)(\phi)}\leqq C\abs{\phi}_{W_{q,-\nu}^{-\ell}(\R;Y')}\abs{u-v}_{W_{p,\nu}^{k}(\R;X)}.
\]
 \end{itemize} For an eventually $(k,\ell)$-Lipschitz continuous
mapping $F$, we denote by $F_{\nu}$ its Lipschitz continuous extension
from $W_{p,\nu}^{k}(\R;X)$ to $W_{p,\nu}^{\ell}(\R;Y)$. Moreover,
denote by $\lip{F_{\nu}}$ the infimum over all possible Lipschitz
constants for $F_{\nu}$. We call $F$ \emph{eventually $(k,\ell)$-contracting}
if $\limsup_{\nu\to\infty}\lip{F_{\nu}}<1$. \end{defn}

\begin{thm}[$L_p$-solution theory] Let $X$ be reflexive and let
$F:C_{c}^{\infty}(\R;X)\to C_{c}^{\infty,+}(\R;X')'$ be $(0,-1)$-contracting.
Then, for $\nu$ large enough, the equation 
\[
\partial_{\nu}u=F_{\nu}(u)
\]
 admits a unique solution $u\in L_{p,\nu}(\R;X)$. \end{thm} \begin{proof}
This result is a special case of Theorem \ref{thm:bst}. \end{proof}

\begin{rem} The analogous result also holds for $(1,0)$-contracting
mappings. Moreover, since the norm of $\partial_{\nu}^{-1}$ as a
mapping from $L_{p,\nu}$ into itself is bounded by $\frac{1}{\nu}$,
we also get a solution theory for $(0,0)$- and $(1,1)$-Lipschitz
mappings, cf. \cite[Corollary 3.4]{Kalauch}. \end{rem}

As causality is a characterizing feature of time-evolution, we are
particularly interested in establishing causality of the solution
operator. 

\begin{defn}[Causality]\label{As-causality-is} Let $E,F$ be vector
spaces. A mapping $W:D(W)\subseteqq E^{\R}\to F^{\R}$ is called \emph{causal}
if for all $x,y\in D(W)$ and $t\in\R$ we have 
\[
\chi_{\R_{<t}}(m)(x-y)=0\Rightarrow\chi_{\R_{<t}}(m)(W(x)-W(y))=0.
\]
 \end{defn} Similar to \cite[Definition 4.3]{Kalauch}, we have to
define a notion of distributional integrals or distributional convolutions.
\begin{defn} Let $w\in C_{c}^{\infty,+}(\R;X)'$. Then we define
\[
\1_{[0,\infty)}\ast w:C_{c}^{\infty,+}(\R;X)\to\K,\phi\mapsto w(\1_{(-\infty,0]}\ast\phi).
\]
 \end{defn} \begin{rem} Assume that $X$ is reflexive. For $w\in W_{p,\nu}^{-1}(\R;X)$
we have that $\1_{[0,\infty)}\ast w=\partial_{\nu}^{-1}w$. Indeed,
by Theorem \ref{thm:dualspace}, we have $w\in W_{q,-\nu}^{1}(\R;X')'$
and thus, for $\phi\in C_{c}^{\infty,+}(\R;X')$, we get that 
\[
\1_{[0,\infty)}\ast w(\phi)=w(\1_{(-\infty,0]}\ast\phi)=w(-\partial_{-\nu}^{-1}\phi)=\langle w,-\partial_{-\nu}^{-1}\phi\rangle_{0,0}=\langle\partial_{\nu}^{-1}w,\phi\rangle_{0,0}=\partial_{\nu}^{-1}w(\phi).
\]
 \end{rem} \begin{thm}[Causality]\label{thm:causality} Assume that
$X$ is reflexive. Let $F:C_{c}^{\infty}(\R;X)\to C_{c}^{\infty,+}(\R;X')'$
be $(0,-1)$ contracting. Then $\partial_{\nu}^{-1}F$ is causal.
\end{thm} \begin{proof} The proof follows essentially along the
lines of \cite[Theorem 4.5]{Kalauch}. Since we are, however, dealing
with a Banach space setting here, the arguments are more delicate
and thus worth recalling in detail. Let $t\in\R$, $\nu_{1}$ such
that $\lip{F_{\nu}}<1$ for all $\nu\geqq\nu_{1}$. Let $\phi\in C^{\infty}(\R)$
be bounded. For $\nu\geqq\nu_{1}$ and $v\in C_{c}^{\infty}(\R;X)$
and $\psi\in C_{c}^{\infty,+}(\R;X')$ with $\sup\supp\psi\leqq t$
we compute 
\begin{align*}
\abs{\partial_{\nu_{1}}^{-1}F_{\nu_{1}}(v)(\psi)-\partial_{\nu_{1}}^{-1}F_{\nu_{1}}(\phi(m)v)(\psi)} & =\abs{\1_{[0,\infty)}\ast F(v)(\psi)-\1_{[0,\infty)}\ast F(\phi(m)v)(\psi)}\\
 & =\abs{\partial_{\nu}^{-1}F_{\nu}(v)(\psi)-\partial_{\nu}^{-1}F_{\nu}(\phi(m)v)(\psi)}\\
 & =\abs{F_{\nu}(v)(-\partial_{-\nu}^{-1}\psi)-F_{\nu}(\phi(m)v)(-\partial_{-\nu}^{-1}\psi)}\\
 & \leq\abs{-\partial_{-\nu}^{-1}\psi}_{W_{q,-\nu}^{1}(\R;X')}\abs{v-\phi(m)v}_{L_{p,\nu}(\R;X)}.\\
 & =\abs{\psi}_{L_{q,-\nu}(\R;X')}\abs{v-\phi(m)v}_{L_{p,\nu}(\R;X)}\\
 & \leq\abs{\psi}_{L_{q,0}(\R;X')}e^{\nu t}\abs{v-\phi(m)v}_{L_{p,\nu}(\R;X)}.
\end{align*}
 By continuity, we deduce that 
\begin{multline*}
\abs{\partial_{\nu_{1}}^{-1}F_{\nu_{1}}(v)(\psi)-\partial_{\nu_{1}}^{-1}F_{\nu_{1}}(\chi_{\R_{<t}}(m)v)(\psi)}\leqq\abs{\psi}_{L_{q,0}(\R;X')}e^{\nu t}\left(\int_{t}^{\infty}\abs{v(\tau)}^{p}e^{-p\nu\tau}\dd\tau\right)^{\frac{1}{p}}\\
=\abs{\psi}_{L_{q,0}(\R;X')}\left(\int_{0}^{\infty}\abs{v(\tau+t)}^{p}e^{-p\nu\tau}\dd\tau\right)^{\frac{1}{p}}.
\end{multline*}
 Hence, letting $\nu\to\infty$, we get the assertion. \end{proof}

\begin{thm}[Independence of $\nu$]Assume that $X$ is reflexive.
Let $F:C_{c}^{\infty}(\R;X)\to C_{c}^{\infty,+}(\R;X')'$ be $(0,-1)$-contracting.
Let $\nu_{1}\in\R_{>0}$ be such that $\lip{F_{\nu}}<1$ for all $\nu\geqq\nu_{1}$.
Let $\nu_{2}\geqq\nu_{1}$. Then, if $w_{\nu_{1}},w_{\nu_{2}}$ satisfy
\[
\partial_{\nu_{1}}w_{\nu_{1}}=F_{\nu_{1}}(w_{\nu_{1}})\text{ and }\partial_{\nu_{2}}w_{\nu_{2}}=F_{\nu_{2}}(w_{\nu_{2}}),
\]
we have $w_{\nu_{1}}=w_{\nu_{2}}$. \end{thm} 

\begin{proof} The proof follows the ideas of the proof of \cite[Theorem 4.6]{Kalauch}:
Let $t\in\R$, $\nu\in\mathbb{R}_{\geqq\nu_{1}}$. Denoting by $w_{\nu}$
the solution of 
\[
\partial_{\nu}w_{\nu}=F_{\nu}(w_{\nu})\in W_{p,\nu}^{-1}(\R;X),
\]
we recall $w_{\nu}\in L_{p,\nu}(\R;X)$. Moreover, we have due to
Theorem \ref{thm:causality} 
\begin{eqnarray*}
\chi_{\R_{<t}}\left(m\right)w_{\nu} & = & \chi_{\R_{<t}}\left(m\right)\partial_{\nu}^{-1}F_{\nu}\left(w_{\nu}\right)\\
 & = & \chi_{\R_{<t}}\left(m\right)\partial_{\nu}^{-1}F_{\nu}\left(\chi_{\R_{<t}}\left(m\right)w_{\nu}\right).
\end{eqnarray*}
 Then, as $\partial_{\nu_{1}}^{-1}F_{\nu_{1}}$ coincides with $\partial_{\nu_{2}}^{-1}F_{\nu_{2}}$
on $C_{c}^{\infty}(\R;X)$ and as an approximating sequence of $C_{c}^{\infty}(\R;X)$-functions
for $\chi_{\R_{<t}}(m_{0})w_{\nu_{2}}$ can be chosen to converge
in both $L_{p,\nu_{1}}(\R;X)$ and $L_{p,\nu_{2}}(\R;X)$, we arrive
at 
\[
\chi_{\R_{<t}}(m)\partial_{\nu_{2}}^{-1}F_{\nu_{2}}(\chi_{\R_{<t}}(m)w_{\nu_{2}})=\chi_{\R_{<t}}(m)\partial_{\nu_{1}}^{-1}F_{\nu_{1}}(\chi_{\R_{<t}}(m)w_{\nu_{2}}).
\]
 Hence, 
\begin{align*}
 & \abs{\chi_{\R_{<t}}(m)(w_{\nu_{1}}-w_{\nu_{2}})}_{L_{p,\nu_{1}}(\R;X)}\\
 & =\abs{\chi_{\R_{<t}}(m)(\partial_{\nu_{1}}^{-1}F_{\nu_{1}}(\chi_{\R_{<t}}(m)w_{\nu_{1}})-\partial_{\nu_{1}}^{-1}F_{\nu_{1}}(\chi_{\R_{<t}}(m)w_{\nu_{2}}))}_{L_{p,\nu_{1}}(\R;X)}\\
 & \leqq\abs{\partial_{\nu_{1}}^{-1}F_{\nu_{1}}(\chi_{\R_{<t}}(m)w_{\nu_{1}})-\partial_{\nu_{1}}^{-1}F_{\nu_{1}}(\chi_{\R_{<t}}(m)w_{\nu_{2}})}_{L_{p,\nu_{1}}(\R;X)}\\
 & \leqq\lip{F_{\nu_{1}}}\abs{\chi_{\R_{<t}}(m)(w_{\nu_{1}}-w_{\nu_{2}})}_{L_{p,\nu_{1}}(\R;X)}.
\end{align*}
 Since $\lip{F_{\nu_{1}}}<1$ the assertion follows. \end{proof}

\subsection{Examples of admissible delay differential equations}

Before we illustrate the applicability of our abstract theorems, we
introduce the notion of having delay and of being amnesic for mappings
from function spaces into function spaces. \begin{defn} Let $E,F$
be vector spaces. A mapping $W:D(W)\subseteqq E^{\R}\to F^{\R}$ is
called \emph{amnesic} if for all $t\in\R$, $x,y\in D(W)$ we have
\[
\chi_{\R_{>t}}(m)(x-y)=0\Rightarrow\chi_{\R_{>t}}(m)(W(x)-W(y))=0
\]
 $W$ is said to \emph{have delay} if $W$ is not amnesic. \end{defn}
We shall give some examples of mappings having delay. \begin{example}[Discrete delay]
For $\theta\in\R$ we define $\tau_{\theta}:L_{p,\nu}(\R;X)\to L_{p,\nu}(\R;X),f\mapsto(t\mapsto f(t+\theta))$.
It is easy to see, that $\tau_{\theta}$ is not causal for $\theta>0$,
whereas it is amnesic. For $\theta<0$, $\tau_{\theta}$ is causal
and has delay. For convenience, we compute the operator norm of $\tau_{\theta}$.
For $f\in C_{c}^{\infty}(\R;X)$, we have 
\[
\abs{\tau_{\theta}f}_{p,\nu}^{p}=\int_{\R}\abs{f(t+\theta)}_{X}^{p}e^{-p\nu t}\dd t=\int_{\R}\abs{f(t+\theta)}_{X}^{p}e^{-p\nu(t+\theta)}\dd t\, e^{p\nu\theta}=\abs{f}_{p,\nu}^{p}e^{p\nu\theta}.
\]
 Thus, $\Abs{\tau_{\theta}}=e^{\nu\theta}$. \end{example}

\begin{example}[Continuous delay]\label{Continuous-Delay} The mapping
$\Theta:L_{p,\nu}(\R;X)\to L_{p,\nu}(\R;L_{p}(\R_{<0};X))$\\
 $\phi\mapsto\phi_{(\cdot)}:=(t\mapsto(\theta\mapsto\phi(t+\theta)))$
clearly has delay. We compute its operator norm. For $f\in C_{c}^{\infty}(\R;X)$,
we have 
\begin{align*}
\abs{\Theta f}_{p,\nu}^{p} & =\int_{\R}\int_{\R_{<0}}\abs{f(t+\theta)}_{X}^{p}\dd\theta e^{-p\nu t}\dd t\\
 & =\int_{\R_{<0}}\int_{\R}\abs{f(t+\theta)}_{X}^{p}e^{-p\nu(t+\theta)}\dd t\, e^{p\nu\theta}\dd\theta\\
 & =\int_{\R_{<0}}\abs{f}_{p,\nu}^{p}e^{p\nu\theta}\dd\theta\\
 & =\abs{f}_{p,\nu}^{p}\frac{1}{p\nu}.
\end{align*}
 Hence, $\Abs{\Theta}=\frac{1}{\sqrt[p]{p\nu}}$. \end{example}

To incorporate initial value problems, it is convenient to have an
adapted point trace result. 

\begin{thm}[Sobolev embedding]\label{thm:sobem} For $\nu\in\R\setminus\{0\}$,
define 
\begin{multline*}
C_{\nu}(\R;X):=\left\{ f:\R\to X;f\text{ continuous,}\phantom{e^{-\nu t}}\right.\\
\left.\abs{f}_{\nu,\infty}:=\sup\{\abs{e^{-\nu t}f(t)}_{X};t\in\R\}<\infty,\ e^{-\nu t}f(t)\to0(t\to\pm\infty)\right\} .
\end{multline*}
 We endow $C_{\nu}(\R;X)$ with the norm $\abs{\cdot}_{\nu,\infty}$
such that it becomes a Banach space. The mapping 
\[
\iota:C_{c}^{\infty}(\R;X)\subseteqq W_{p,\nu}^{1}(\R;X)\to C_{\nu}(\R;X),f\mapsto f
\]
 is continuous. \end{thm} \begin{proof} We shall only prove the
case $\nu>0$. The case $\nu<0$ can be dealt with similarly. Let
$f\in C_{c}^{\infty}(\R;X)$ and $s,t\in\R,s<t$. Then 
\begin{align*}
\abs{f(t)-f(s)} & \leqq\int_{s}^{t}\abs{\partial_{\nu}f(\xi)}\dd\xi\\
 & =\int_{s}^{t}\abs{\partial_{\nu}f(\xi)}e^{-\nu\xi}e^{\nu\xi}\dd\xi\\
 & \leqq\left(\int_{s}^{t}\abs{\partial_{\nu}f(\xi)}^{p}e^{-p\nu\xi}\dd\xi\right)^{\frac{1}{p}}\left(\int_{s}^{t}e^{q\nu\xi}\dd\xi\right)^{\frac{1}{q}}\\
 & \leqq\abs{f}_{1,\nu,p}\left(\frac{1}{q\nu}\left(e^{q\nu t}-e^{q\nu s}\right)\right)^{\frac{1}{q}}.
\end{align*}
 Letting $s\to-\infty$ in this inequality, we arrive at 
\[
e^{-\nu t}\abs{f(t)}\leqq\frac{1}{\sqrt[q]{q\nu}}\abs{f}_{1,\nu,p},
\]
 which gives the continuity result. \end{proof}

By the aforementioned theorem, $\iota$ can be uniquely continuously
extended to $W_{p,\nu}^{1}(\R;X)$. As the extension of $\iota$ is
the extension of the identity mapping, we omit $\iota$ in the following,
and choose, without giving explicit reference to it, the continuous
representer of a $W_{p,\nu}^{1}(\R;X)$-function. (It is easy to see
that the continuous extension is one-to-one.)

Now, we have all the tools at hand to apply our general solution theory
to a number of example cases.

\begin{example}[Initial value problems, cf. {\cite[Theorem 5.4]{Kalauch}}]
Let $\nu>0$, $u_{0}\in X$. Let $F$ be $(0,0)$-Lipschitz and such
that $F(\phi)=0$ for all $\phi\in C_{c}^{\infty}(\R;X)$ with $\supp\phi\subseteqq(-\infty,0]$.
Then the equation%
\footnote{By Theorem \ref{thm:sobem}, we have that the point evaluation at
$0$, denoted by $\delta$, is an element of $W_{p,\nu}^{-1}$. For
a Banach space element $u_{0}$ we write $\delta u_{0}$ for the derivative
of the map $t\mapsto\1_{[0,\infty)}(t)u_{0}$. Thus, in this sense
it holds $\delta u_{0}\in W_{p,\nu}^{-1}(\R;X)$.%
} 
\[
\partial_{\nu}u=F_{\nu}(u)+\delta u_{0}
\]
admits a unique solution $u\in L_{p,\nu}(\R;X)$ and such that $u-\chi_{\R_{>0}}(m)u\in W_{p,\nu}^{1}(\R;X)$
and $u(0+)=u_{0}$ if $\nu$ is chosen sufficiently large. 

Unique existence of $u$ follows from our general solution theory.
The remaining facts follow from the representation 
\[
u-\chi_{\R_{>0}}(m)u=u-\partial_{\nu}^{-1}\delta u_{0}=\partial_{\nu}^{-1}F(u)
\]
 and causality of $\partial_{\nu}^{-1}F_{\nu}$. \end{example}

\begin{example}[Finite discrete delay] Let $\theta_{1},\ldots,\theta_{n}\in\R_{\leqq0}$
be distinct, and let $\Phi:C_{c}^{\infty}(\R;X^{n})\to C_{c}^{\infty,+}(\R;X')'$
be $(0,-1)$-contracting. Then, for $\nu$ sufficiently large, the
equation 
\[
\partial_{\nu}u=\Phi_{\nu}(\tau_{\theta_{1}}u,\cdots,\tau_{\theta_{n}}u)
\]
 admits a unique solution $u\in L_{p,\nu}(\R;X)$.

It suffices to observe that the operator norm of 
\[
\Theta:L_{p,\nu}(\R;X)\to L_{p,\nu}(\R;X^{n}),f\mapsto(\tau_{\theta_{1}}f,\ldots,\tau_{\theta_{n}}f)
\]
 can be estimated arbitrarily close to $1$, if $\nu$ was chosen
sufficiently large. \end{example}

\begin{example}[Continuous delay] Let $\Phi:C_{c}^{\infty}(\R;L_{p}(\R_{<0};X))\to C_{c}^{\infty,+}(\R;X')'$
be $(0,-1)$-Lipschitz. Then, for $\nu$ sufficiently large, the equation
\[
\partial_{\nu}u=\Phi_{\nu}(u_{(\cdot)})
\]
 admits a unique solution, if $\nu$ is chosen sufficiently large.

The assertion follows from the Example \ref{Continuous-Delay}, where
we estimated the operator norm of the mapping $\phi\mapsto\phi_{(\cdot)}$
in the weighted spaces under consideration. \end{example}

\begin{example}[Neutral differential equations] Let $\Phi:C_{c}^{\infty}(\R;L_{p}(\R_{<0};X)^{2})\to C_{c}^{\infty,+}(\R;X')'$
be $(0,0)$-Lipschitz. Then the equation 
\[
\partial_{\nu}u=\Phi(u_{(\cdot)},(\partial_{\nu}u)_{(\cdot)})
\]
admits a unique solution $u\in W_{p,\nu}^{1}(\R;X)$, if $\nu$ was
chosen large enough. 

Consider the mapping 
\[
\Theta:W_{p,\nu}^{1}(\R;X)\to L_{p,\nu}(\R;L_{p}(\R_{<0};X)^{2}),f\mapsto(f_{(\cdot)},(\partial_{\nu}f)_{(\cdot)}).
\]
Note that the operator norm of $W_{p,\nu}^{1}(\R;X)\to L_{p,\nu}(\R;X)^{2},f\mapsto(f,\partial_{\nu}f)$
is bounded if $\nu\to\infty$ and that the mapping $L_{p,\nu}(\R;X)\ni f\mapsto f_{(\cdot)}\in L_{p,\nu}(\R;L_{p}(\R_{<0};X))$
has operator norm tending to $0$ if $\nu\to\infty$, by the aforementioned
example. We deduce that $\Theta$ is eventually $(1,0)$-contracting,
with arbitrarily small operator norm and that the map $\Phi\circ\Theta$
is eventually $(1,0)$-contracting. Hence, our general solution theory
applies. \end{example}

In the following, we will treat some more concrete examples form the
literature.

\begin{example} The following example has been considered in \cite{Corduneanu1973,Corduneanu1975,Havarneanu1981}
and the references therein. Let $B\in L_{1}(\R_{>0};\R^{n\times n})$,
$(A_{j})_{j}\in\ell_{1}(\mathbb{N};\R^{n\times n})$, $(t_{j})_{j}\in\R_{\geq0}^{\mathbb{N}}$
and let $f\in L_{p}(\R;\R^{n})$ be such that the support is bounded
from below. Consider the problem of finding $x\in L_{p,\nu}(\R;\R^{n\times n})$
such that 
\[
\partial_{\nu}x=\sum_{j=0}^{\infty}A_{j}\tau_{-t_{j}}x+B\ast x+f.
\]
 The unique existence of $x$ follows by observing that the operator
\begin{align*}
F & :L_{p,\nu}(\R;\R^{n})\to L_{p,\nu}(\R;\R^{n})\\
 & x\mapsto\left(\sum_{j=0}^{\infty}A_{j}\tau_{-t_{j}}x+B\ast x\right)
\end{align*}
is Lipschitz continuous. Indeed, Young's inequality ensures $\abs{B\ast x}\leq\abs{B}_{L^{1}}\abs{x}$
for all $x\in L_{p,\nu}(\R;\R^{n})$. The first term we estimate as
follows. Let $x\in L_{p,\nu}(\R;\R^{n})$. Then 
\[
\abs{\sum_{j=0}^{\infty}A_{j}\tau_{-t_{j}}x}\leq\sum_{j=0}^{\infty}\abs{A_{j}}\abs{x}=\abs{(A_{j})_{j}}_{\ell^{1}}\abs{x}
\]
 \end{example} 

In \cite{Das2006} the oscillations of possible solutions to the following
problem are discussed. 

\begin{example} Let $k\in\mathbb{N},n\in\mathbb{N}_{>0}$ and for
$j\in\{0,\ldots,k\}$ let $p_{j}:\R\to\R$ be continuous and bounded
and $\sigma_{j}\in\R_{>0}$. Let $f\in L_{p}(\R)$ with support bounded
from below and consider the following neutral differential equation
of $n$'th order 
\[
(x-p_{0}(m)\tau_{-\sigma_{0}}x)^{(n)}=\sum_{j=1}^{k}p_{j}(m)\tau_{-\sigma_{j}}x+f.
\]
 For $\nu\in\R_{>0}$, we may equally discuss 
\[
\partial_{\nu}^{n}(x-p_{0}(m)\tau_{-\sigma_{0}}x)=\sum_{j=1}^{k}p_{j}(m)\tau_{-\sigma_{j}}x+f.
\]
 The latter is the same as 
\[
x=\partial_{\nu}^{-n}\left(\sum_{j=1}^{k}p_{j}(m)\tau_{-\sigma_{j}}x\right)+\partial_{\nu}p_{0}(m)\tau_{-\sigma_{0}}x+\partial_{\nu}^{-n}f.
\]
We observe that this is a fixed point problem, which admits a unique
solution for $\nu$ large enough. Indeed, the operator norm of 
\[
\left(\partial_{\nu}^{-n}\left(\sum_{j=1}^{k}p_{j}(m)\tau_{-\sigma_{j}}\right)+p_{0}(m)\tau_{-\sigma_{0}}\right):L_{p,\nu}(\R)\to L_{p,\nu}(\R)
\]
 can be estimated by 
\begin{align*}
\Abs{\partial_{\nu}^{-n}\left(\sum_{j=1}^{k}p_{j}(m)\tau_{-\sigma_{j}}\right)+p_{0}(m)\tau_{-\sigma_{0}}} & \leq\Abs{\partial_{\nu}^{-n}\left(\sum_{j=1}^{k}p_{j}(m)\tau_{-\sigma_{j}}\right)}+\Abs{p_{0}(m)\tau_{-\sigma_{0}}}\\
 & \leq\frac{1}{\nu^{n}}\sum_{j=1}^{k}\abs{p_{j}}_{\infty}\exp(-\sigma_{j}\nu)+\abs{p_{0}}_{\infty}\exp(-\sigma_{0}\nu),
\end{align*}
which is eventually less than $1$ if $\nu$ is large enough. Note
that for having a solution theory, the continuity of the $p_{j}$'s
was not needed. \end{example} 

A different class of neutral differential equations, which was considered
in \cite{Frasson2005,Hale1970} is as follows. We shall treat the
Hilbert space case here for convenience.

\begin{example} Let $H$ be a Hilbert space and $M,L\in L(L_{2}(\R_{<0};H);H)$.
Consider the following equation 
\begin{equation}
(t\mapsto Mx_{t})'=(t\mapsto Lx_{t})+f,\label{eq:gen:neutr}
\end{equation}
where $f\in L_{2}(\R;H)$ with support bounded from below is given.
Our space-time approach prerequisites the consideration of the operator
\[
\Theta:C_{c}^{\infty}(\R;H)\to C_{c}^{\infty}(\R;L_{2}(\R_{<0};H)),\phi\mapsto\phi_{(\cdot)}
\]
in a slightly different version than before. We note that for $\phi\in C_{c}^{\infty}(\R;H)$
and $\nu\in\R_{>0}$, we have 
\begin{equation}
\partial_{\nu}\Theta\phi=\Theta\partial_{\nu}\phi.\label{eq:ThetaInter}
\end{equation}
For any $\nu\in\R_{>0}$ there is a continuous extension $\Theta_{\nu}$
as a mapping from $L_{2,\nu}(\R;H)$ to $L_{2,\nu}(\R;L_{2}(\R_{<0};H))$.
Moreover, for all $x\in L_{2,\nu}(\R;H)$ we have 
\[
\langle\Theta x,\Theta x\rangle=\frac{1}{2\nu}\langle x,x\rangle.
\]
 This equality yields the closedness of the range of $\Theta_{\nu}$.
Continuous extension of \eqref{eq:ThetaInter} for all $\phi\in W_{2,\nu}^{1}(\R;H)$
yields that $\partial_{\nu}$ leaves $R(\Theta_{\nu})$ invariant.
Furthermore, we have the same property for $\partial_{\nu}^{-1}$.
Hence, our perspective on \eqref{eq:gen:neutr} is the following.
Consider 
\[
\partial_{\nu}M(\Theta_{\nu}x)=L\Theta_{\nu}x+f.
\]
By the continuous invertibility of $\Theta$ and the intertwining
relation \eqref{eq:ThetaInter}, this is the same as to consider 
\[
\partial_{\nu}\Theta_{\nu}M(\Theta_{\nu}x)=\Theta_{\nu}L\Theta_{\nu}x+\Theta_{\nu}f.
\]
 Assume $\mathcal{M}_{\nu}:R(\Theta_{\nu})\to R(\Theta_{\nu}),y\mapsto\Theta_{\nu}My$
to be continuously invertible. Therefore, we formulate the equation
as follows 
\[
\partial_{\nu}\Theta_{\nu}M(y)=\Theta_{\nu}Ly+\Theta_{\nu}f.
\]
 for $y\in R(\Theta_{\nu})$. Now, our general solution theory applies
to the equation 
\[
\partial_{\nu}y=\Theta_{\nu}\left(L\mathcal{M}_{\nu}^{-1}y+f\right).
\]
 This yields a unique solution $y\in R(\Theta_{\nu})$. The solution
of equation \eqref{eq:gen:neutr} is then given by $x=\Theta_{\nu}^{-1}\s M_{\nu}^{-1}y$.
\end{example}

\section{The non-reflexive case -- spaces of continuous functions\label{sec:The-non-reflexive-case}}

In this section, we will describe how to adapt the general solution
theory of Section \ref{sec:gst} to the non-reflexive setting. The
main difficulty to overcome is to give appropriate meaning to ``eventually
$(k,\ell)$-Lipschitz continuous'' in order to state a coherent theory.
For the whole section, let $X$ be a Banach space. We focus here on
the $L_{\infty}$-norm, we could, however, also treat the case of
$L_{1}$-functions. As the case of $L_{1}$ is a hybrid of distributional
derivatives similar to the previous part and the issues resulting
from the non-reflexivity of the underlying space as discussed in the
following sections, we only consider the $L_{\infty}$-norm here.

\subsection{The time-derivative}

The distributional time-derivative as presented in Section \ref{sec:dde-lp}
cannot be used in the straightforward way by choosing $L_{\infty}$
as underlying space, since the (distributional) time-derivative would
not be densely defined anymore. Thus, we consider the more or less
classical way of discussing delay differential equations and use the
space of Banach space valued continuous functions $C_{\nu}(\R;X)$,
which we have already defined in Theorem \ref{thm:sobem}, as the
underlying space.

\begin{defn} For $\nu\in\R$, define $\partial_{\nu}:C_{\nu}^{1}(\R;X)\subseteqq C_{\nu}(\R;X)\to C_{\nu}(\R;X),f\mapsto f'$,
where $C_{\nu}^{1}(\R;X)\coloneqq\{f\in C_{\nu}(\R;X);f'\in C_{\nu}(\R;X)\}$.
\end{defn} 

\begin{prop}\label{prop:nu} Let $\nu\in\R\setminus\{0\}$. Then
$0\in\rho(\partial_{\nu})$, $\partial_{\nu}^{-1}f(t)=\int_{-\infty}^{t}f(\tau)\dd\tau$
($t\in\R$, $\nu>0$) and $\Abs{\partial_{\nu}^{-1}}=\frac{1}{\abs{\nu}}$.
\end{prop} \begin{proof} For $f\in C_{c}^{\infty}(\R;X)$ and $\nu>0$,
we compute 
\begin{align*}
\abs{e^{-\nu t}\int_{-\infty}^{t}f(\tau)\dd\tau} & =\abs{\int_{-\infty}^{t}e^{-\nu t+\nu\tau}f(\tau)e^{-\nu\tau}\dd\tau}\\
 & \leqq\int_{-\infty}^{t}e^{-\nu t+\nu\tau}\dd\tau\abs{f}_{C_{\nu}(\R;X)}=\frac{1}{\nu}\abs{f}_{C_{\nu}(\R;X)}.
\end{align*}
In order to see the remaining inequality, for $n\in\N$ take a function
$\phi_{n}\in C_{c}^{\infty}(\R)$ such that $0\leqq\phi_{n}\leqq1$
and $\phi_{n}=1$ on $[-n,n]$. Let $x\in X$ with $\abs{x}=1$ and
define $f_{n}(t)\coloneqq e^{\nu t}\phi_{n}(t)x$ for $n\in\N,$$t\in\R$.
Note that $\abs{f_{n}}_{C_{\nu}(\R;X)}\leqq1$ for all $n\in\N.$
Moreover, observe that for $n\in\N$ we have
\begin{align*}
\sup\{\abs{e^{-\nu t}(\partial_{\nu}^{-1}f_{n})(t)};t\in\R\} & \geqq e^{-\nu n}\int_{-n}^{n}e^{\nu\tau}\dd\tau\\
 & =\frac{1}{\nu}e^{-\nu n}\left(e^{\nu n}-e^{-\nu n}\right)\\
 & =\frac{1}{\nu}\left(1-e^{-2\nu n}\right)\to\frac{1}{\nu}\quad(n\to\infty).
\end{align*}
This yields $\Abs{\partial_{\nu}^{-1}}\geqq\frac{1}{\nu}$. The case
$\nu<0$ is similar. \end{proof}

Hence, $\partial_{\nu}$ is a possible choice for $C$ in the basic
solution theory. Before, we state the solution theory also in this
case, we define eventually Lipschitz continuous mappings to have a
prototype of right-hand sides at hand. We denote $C_{\nu}^{k}(\R;X):=\s X_{k}(\partial_{\nu})$
for $k\in\{1,0,-1\}$. Due to the non-reflexivity of $C_{\nu}(\R;X)$,
we cannot define eventual Lipschitz continuity for mappings with values
in a space of linear functionals. Instead of characterizing the negative
extra\-po\-la\-tion spaces as suitable duals, we introduce the
space $C_{-\infty}(\R;X):=\bigcup_{\nu\in\mathbb{R}_{>0}}C_{\nu}^{-1}(\R;X)$.
In order to compare elements of ``negative'' spaces for different
parameters $\nu$, we define the following equality relation between
these elements: For $\phi\in C_{\nu}^{-1}(\R;X)$ and $\psi\in C_{\mu}^{-1}(\R;X)$
we define 
\[
\phi=\psi:\Leftrightarrow\partial_{\nu}^{-1}\phi=\partial_{\mu}^{-1}\psi.
\]

\begin{rem} Let $\phi\in C_{\nu}^{-1}(\mathbb{R};X),\psi\in C_{\mu}^{-1}(\mathbb{R};X)$
with $\phi=\psi$. Then there exists a sequence $(\rho_{n})_{n\in\mathbb{N}}$
in $C_{c}^{\infty}(\mathbb{R};X)$ such that $\rho_{n}\to\phi$ in
$C_{\nu}^{-1}(\mathbb{R};X)$ and $\rho_{n}\to\psi$ in $C_{\mu}^{-1}(\R;X)$
as $n\to\infty$. Indeed, let $(\gamma_{n})_{n\in\mathbb{N}}\in C_{c}^{\infty}(\mathbb{R})^{\mathbb{N}}$
be a mollifier and define $\tilde{\rho}_{n}:=\gamma_{n}\ast\partial_{\nu}^{-1}\phi\in C_{\nu}(\R;X)\cap C_{\mu}(\R;X)$.
Then we obtain, due to the continuity of the translation operator
\[
[0,1]\ni s\mapsto(f\mapsto f(\cdot+s))\in L(C_{\nu}(\R;X))
\]
 for each $\nu\in\mathbb{R}_{>0}$, that $\tilde{\rho}_{n}\to\partial_{\nu}^{-1}\phi$
in $C_{\nu}(\R;X)$ and $C_{\mu}(\R;X)$ as $n\to\infty$. For $k\in\mathbb{N}$
let now $\chi_{k}\in C_{c}^{\infty}(\mathbb{R})$ such that $0\leq\chi_{k}\leq1$
and $\chi_{k}=1$ on $(-k,k)$. Then an easy computation shows $\chi_{k}\tilde{\rho}_{n}\to\tilde{\rho}_{n}$
in $C_{\nu}(\R;X)$ and $C_{\mu}(\R;X)$ as $k\to\infty$. Hence,
we find a strictly increasing sequence $(k_{n})_{n}$ of integers
such that $(\rho_{n})_{n}:=((\chi_{k_{n}}\tilde{\rho}_{n})')_{n}$
has the desired properties. \end{rem}

\begin{defn}[eventually $(k,\ell)$-contracting] Let $k,\ell\in\{1,0,-1\}$
and let $Y$ be a Banach space. A mapping $F:C_{c}^{\infty}(\R;X)\to C_{-\infty}(\R;Y)$
is called \emph{eventually $(k,\ell)$-Lipschitz continuous} if the
following assumptions are satisfied: \begin{itemize} 

\item there exists $\nu_{0}>0$ such that for all $\nu\geqq\nu_{0}$
we have $F(0)\in C_{\nu}^{-1}(\R;Y)$, 

\item there exists $\nu_{1}>0$ and $C>0$ such that for all $\nu\geqq\nu_{1}$,
$u,v\in C_{c}^{\infty}(\R;X)$ we have 
\[
\abs{F(u)-F(v)}_{C_{\nu}^{\ell}(\R;Y)}\leqq C\abs{u-v}_{C_{\nu}^{k}(\R;X)}.
\]
 \end{itemize} For an eventually $(k,\ell)$-Lipschitz continuous
mapping $F$, we denote by $F_{\nu}$ its Lipschitz continuous extension
from $C_{\nu}^{k}(\R;X)$ to $C_{\nu}^{\ell}(\R;Y)$. Moreover, denote
by $\lip{F_{\nu}}$ the infimum over all possible Lipschitz constants
for $F_{\nu}$. We call $F$ \emph{eventually $(k,\ell)$-contracting}
if $\limsup_{\nu\to\infty}\lip{F_{\nu}}<1$. \end{defn}

\begin{rem} Note that for a $(k,\ell)$-Lipschitz continuous mapping
we have $F[C_{c}^{\infty}(\mathbb{R};X)]\subseteqq\bigcap_{\nu\geqq\nu_{0}}C_{\nu}^{-1}(\R;X)$
for some $\nu_{0}>0$, where the intersection is understood with respect
to the equality relation defined above. \end{rem}

\begin{thm}[Solution theory]{\label{thm:sol_theory_cont}} Let
$F:C_{c}^{\infty}(\R;X)\to C_{-\infty}(\R;X)$ be $(0,-1)$-contracting.
Then there exists a unique solution $u\in C_{\nu}(\R;X)$ of the equation
\[
\partial_{\nu}u=F_{\nu}(u)
\]
 if $\nu$ is chosen sufficiently large. \end{thm} \begin{proof}
Clear. \end{proof}

\begin{rem} The latter theorem also extends to the case of $(1,0)$-contracting.
Moreover, similar to Section \ref{sec:dde-lp} and due to Proposition
\ref{prop:nu}, we also have a solution theory for $(0,0)$- or $(1,1)$-Lipschitz
continuous mappings, which is the common situation. \end{rem}

\begin{thm}[Causality] Let $F:C_{c}^{\infty}(\R;X)\to C_{-\infty}(\R;X)$
be $(0,-1)$-contracting. Then $\partial_{\nu}^{-1}F_{\nu}$ is causal
if $\nu$ is chosen sufficiently large. \end{thm} \begin{proof}
Let $\nu_{0}\in\R_{>0}$ be such that $F$ admits a Lipschitz-continuous
extension for all $\nu\geqq\nu_{0}$. Let $\tau\in\R$ and let $\phi\in C^{\infty}(\R)$
be such that $0\leqq\phi\leqq1$, $\phi(s)=0$ for $s\geqq\tau$ and
$\phi(t)=1$ for $t\leqq\tau-\varepsilon$ for some $\varepsilon\in\R_{>0}$.
We show that $\partial_{\nu}^{-1}F_{\nu}(v)(t)=\partial_{\nu}^{-1}F_{\nu}(\phi(m)v)(t)$
for $v\in C_{c}^{\infty}(\R;X)$ and $t\leqq\tau-\eps$. Let $\psi\in C_{c}^{\infty}(\R;X')$
be such that $\sup\supp\psi\leqq\tau-\varepsilon$. We compute for
$v\in C_{c}^{\infty}(\R;X)$ and $\eta\geqq\nu$ 
\begin{align*}
 & \int_{\R}\abs{\langle\partial_{\nu}^{-1}F_{\nu}(v)-\partial_{\nu}^{-1}F_{\nu}(\phi(m)v),\psi\rangle}\\
 & \leqq\int_{\R}\abs{\partial_{\nu}^{-1}F_{\nu}(v)-\partial_{\nu}^{-1}F_{\nu}(\phi(m)v)}\abs{\psi}=\int_{\R}\abs{\partial_{\eta}^{-1}F_{\eta}(v)-\partial_{\eta}^{-1}F_{\eta}(\phi(m)v)}\abs{\psi}\\
 & \leqq\abs{\partial_{\eta}^{-1}F_{\eta}(v)-\partial_{\eta}^{-1}F_{\eta}(\phi(m)v)}_{\infty,\eta}\int_{-\infty}^{\tau-\varepsilon}\abs{\psi(t)}e^{\eta t}\dd t\\
 & \leqq\abs{(1-\phi(m))v}_{\infty,\eta}\int_{-\infty}^{\tau-\varepsilon}\abs{\psi(t)}e^{\eta t}\dd t\\
 & \leqq\sup\{\abs{e^{-\eta t}v(t)};\tau-\varepsilon\leqq t<\infty\}\int_{-\infty}^{\tau-\varepsilon}\abs{\psi(t)}e^{\eta t}\dd t\\
 & =\sup\{\abs{e^{-\eta(t+\tau-\varepsilon)}v(t+\tau-\varepsilon)};0\leqq t<\infty\}\int_{-\infty}^{0}\abs{\psi(t+\tau-\varepsilon)}e^{\eta(t+\tau-\varepsilon)}\dd t\\
 & =\sup\{\abs{e^{-\eta t}v(t+\tau-\varepsilon)};0\leqq t<\infty\}\int_{-\infty}^{0}\abs{\psi(t+\tau-\varepsilon)}e^{\eta t}\dd t\\
 & \to0\quad(\eta\to\infty),
\end{align*}
 where in the third line we have used the definition of equality of
elements in $C_{\nu}^{-1}(\R;X)$ and $C_{\eta}^{-1}(\R;X)$. Thus,
\[
\sup\{\abs{\partial_{\nu}^{-1}F_{\nu}(v)(t)-\partial_{\nu}^{-1}F_{\nu}(\phi(m)v)(t)e^{-\nu t}};-\infty\leqq t\leqq\tau-\varepsilon\}=0.
\]
 This yields causality. \end{proof}

\begin{thm}[Independence of $\nu$ ] Let $F:C_{c}^{\infty}(\R;X)\to C_{-\infty}(\R;X)$
be $(0,-1)$-contracting and $\nu_{0}>0$ such that $\abs{F_{\nu}}_{\mathrm{Lip}}<1$
for each $\nu\geqq\nu_{0}$. Let $\nu\geqq\mu\geqq\nu_{0}$ and let
$v_{\nu}\in C_{\nu}(\R;X),\, v_{\mu}\in C_{\mu}(\R;X)$ denote the
solutions of the equations 
\[
\partial_{\nu}v_{\nu}=F_{\nu}(v_{\nu})\mbox{ and }\partial_{\mu}v_{\mu}=F_{\mu}(v_{\mu}),
\]
 respectively. Then $v_{\nu}=v_{\mu}$. \end{thm}

\begin{proof} Let $t\in\mathbb{R}$ and let $\phi\in C^{\infty}(\R)$
be such that $0\leqq\phi\leqq1$ and $\phi(s)=0$ for $s\geqq t$
and $\phi(s)=1$ for $s\leqq t-\eps$ for some $\eps>0$. Note that
for $w\in C_{\nu}(\R;X)$ we have $\phi(m)w\in C_{\mu}(\R;X)$ with
$|\phi(m)w|_{\mu,\infty}\leq e^{(\nu-\mu)t}|\phi(m)w|_{\nu,\infty}$.
Hence, $\partial_{\nu}^{-1}F_{\nu}(\phi(m)v_{\nu})=\partial_{\mu}^{-1}F_{\mu}(\phi(m)v_{\nu})$,
since we can approximate $\phi(m)v_{\nu}$ by the same sequence of
test functions in both spaces $C_{\nu}(\R;X)$ and $C_{\mu}(\R;X)$.
Hence, we obtain by using the causality of $\partial_{\nu}^{-1}F_{\nu}$
\begin{align*}
|\chi_{\R_{<t-\eps}}(m)(v_{\nu}-v_{\mu})|_{\infty,\mu} & =|\chi_{\R_{<t-\eps}}(m)(\partial_{\nu}^{-1}F_{\nu}(\phi(m)v_{\nu})-\partial_{\mu}^{-1}F_{\mu}(\phi(m)v_{\mu}))|_{\infty,\mu}\\
 & =|\chi_{\R_{<t-\eps}}(m)(\partial_{\mu}^{-1}F_{\mu}(\phi(m)v_{\nu})-\partial_{\mu}^{-1}F_{\mu}(\phi(m)v_{\mu}))|_{\infty,\mu}\\
 & \leq\abs{F_{\mu}}_{\mathrm{Lip}}|\phi(m)v_{\nu}-\phi(m)v_{\mu}|_{\infty,\mu}.\\
 & \leq\abs{F_{\mu}}_{\mathrm{Lip}}|\chi_{\R_{<t}}(m)(v_{\nu}-v_{\mu})|_{\infty,\mu}.
\end{align*}
 Thus, we get $\chi_{\R_{<t}}(m)v_{\nu}=\chi_{\R_{<t}}(m)v_{\mu}$
for each $t\in\R$ and hence, $v_{\nu}=v_{\mu}$. \end{proof}

\subsection{Examples\label{sub:Examples}}

Let us describe the space $C_{-\infty}(\R;\R):=C_{-\infty}(\R)$ in
more detail. Let $\mu$ be a Borel measure on $\R$ such that for
some $\nu>0$ we have $t\mapsto\mu((-\infty,t])\in C_{\nu}(\R)$.
Then $\mu\in C_{\nu}^{-1}(\R)$. Indeed, let $(g_{n})_{n}$ be a $C_{c}^{\infty}(\R)$
sequence approximating $\mu((-\infty,\cdot])$ in $C_{\nu}(\R)$.
Then $g_{n}'$ converges to $\mu$ in $C_{\nu}^{-1}(\R)$. Moreover,
the derivative applied to $\mu((-\infty,\cdot])$ is just the distributional
derivative: Let $\phi\in C_{c}^{\infty}(\R)$ then we have for $n\in\N$
\[
\int_{\R}g_{n}'\phi=-\int_{\R}g_{n}\phi'\to-\int_{\R}\int_{-\infty}^{t}\dd\mu(s)\phi'(t)\dd t=-\int_{\R}\int_{s}^{\infty}\phi'(t)\dd t\,\dd\mu(s)=\int_{\R}\phi\,\dd\mu.
\]
 A particular instance of such measures are bounded measures with
support bounded below and a continuous cumulative distribution function.
As a non-trivial example we mention the derivative of the ``devil's
staircase''.

Another example is the derivative of the function $f:x\mapsto\1_{[0,\infty)}\cos(\frac{\pi}{x})x$.
Since $f$ is of unbounded variation, its derivative ($x\mapsto\cos(\frac{\pi}{x})+\frac{\sin(\frac{\pi}{x})}{x}\pi$)
is not a Borel measure.

\begin{example}[IVP for ODE]\label{IVP for Ode} Let $F:D(F)\subseteqq X^{\R}\to C_{-\infty}(\R;X)$
with $F(\phi)=0$ for each $\phi\in C_{c}^{\infty}(\R;X)(\subseteqq D(F))$
with $\supp\phi\subseteqq(-\infty,0)$. We assume that for every $x\in X$
the mapping 
\begin{align*}
G_{x}:C_{c}^{\infty}(\R;X) & \to C_{-\infty}(\R;X)\\
\phi & \mapsto F(\phi+\1_{[0,\infty)}x)
\end{align*}
is $(0,-1)$-contracting. Then, by Theorem \ref{thm:sol_theory_cont},
we find for every $u^{(0)}\in X$ a unique solution $v\in C_{\nu}(\R;X)$
of the equation 
\begin{align*}
\partial_{\nu}v=G_{u^{(0)}}(v).
\end{align*}
Moreover, due to causality, $v$ is supported on $[0,\infty)$. We
define $u:=v+\1_{[0,\infty)}u^{(0)}\in C_{\nu}^{-1}(\R;X)$, which
solves the equation $\dot{u}=F(u)$ on the half axis $\mathbb{R}_{>0}$.
Furthermore, $v=u-\1_{[0,\infty)}u^{(0)}$ is continuous and thus,
$0=v(0-)=v(0+)=u(0+)-u^{(0)}$, which gives that $u$ satisfies the
initial condition $u(0+)=u^{(0)}$. \end{example}

\begin{rem}[Nemitzkii-operator] The mapping $F:\phi\mapsto(t\mapsto f(\phi(t)))$
for some Lipschitz continuous $f:X\to X$ satisfies the assumptions
from the previous example. \end{rem}

\begin{example} Let $H$ be a Hilbert space. Let $T,K\in L(H)$.
Consider the initial value problem: 
\[
\begin{cases}
\dot{S}(t)=TS(t)-S(t)T & ,t>0,\\
S(0)=K.
\end{cases}
\]
 The latter equation has a unique continuous solution $S:\R_{\geqq0}\to L(H)$.
Indeed, consider the operator 
\[
[\cdot,T]:L(H)\to L(H),S\mapsto TS-ST.
\]
 The latter is a continuous mapping with norm bounded by $2\Abs{T}$.
Considering $G_{K}$ as above with $F(\phi):=(t\mapsto[\phi(t),T])$
for functions $\phi:\R\to L(H)$, we are in the situation of Example
\ref{IVP for Ode}. \end{example}

Similarly to Section \ref{sec:dde-lp}, we are now in the position
to discuss several delay type equations. For sake of brevity we shall
only list the operators involved and compute their operator norms.

\begin{example} The operator $\tau_{\theta}:C_{\nu}(\R;X)\to C_{\nu}(\R;X),f\mapsto f(\cdot+\theta)$
has operator norm $e^{\nu\theta}$, which can be read off from the
following. From 
\[
e^{-\nu t}f(t+\theta)=e^{\nu\theta}e^{-\nu(t+\theta)}f(t+\theta)
\]
for $t\in\R$ and $f\in C_{c}^{\infty}(\R;X)$, we see that $\Abs{\tau_{\theta}}=e^{\nu\theta}$.
\end{example}

\begin{example} The operator $C_{\nu}(\R;X)\ni\phi\mapsto\phi_{(\cdot)}\in C_{\nu}(\R;C_{b}(\R_{<0};X))$
has norm bounded by $1$. Indeed, for $t\in\R$ and $\phi\in C_{c}^{\infty}(\R;X)$
we compute 
\[
\abs{\phi_{(t)}}_{C_{b}(\R_{<0};X)}=\sup_{\theta\in\R_{<0}}\abs{\phi(t+\theta)}=\sup_{\theta\in\R_{<0}}\abs{\phi(t+\theta)e^{-\nu(t+\theta)}}e^{\nu(t+\theta)}\leqq\abs{\phi}_{\nu,\infty}e^{\nu t}.
\]
 \end{example}

\begin{rem} Note that $\phi\mapsto\phi_{(\cdot)}$ is not a strict
contraction for $\nu$ large as it has been in the $L_{p}$-case.
Hence, in order to solve equations of the form 
\[
\partial_{\nu}u=F_{\nu}(u_{(\cdot)})
\]
$F_{\nu}$ needs to be $(0,-1)$-contracting. So, Lipschitz-continuity
does not suffice to establish a well-posedness theorem, at least in
the continuous case. Moreover, note that this perspective also effects
neutral differential equations of the form $\partial_{\nu}u=F_{\nu}(u_{(\cdot)},(\partial_{\nu}u)_{(\cdot)})$
for suitable $F$. In that case one has to assume that $F$ is $(0,0)$-contracting
and not only $(0,0)$-Lipschitz. \end{rem}

\end{document}